\documentclass[12pt, a4paper]{amsart}

\let\oldtocsection=\tocsection
\let\oldtocsubsection=\tocsubsection
\let\oldtocsubsubsection=\tocsubsubsection
\renewcommand{\tocsection}[2]{\hspace{0pt}\oldtocsection{#1}{#2}}
\renewcommand{\tocsubsection}[2]{\hspace{16pt} \oldtocsubsection{#1}{#2} }
\renewcommand{\tocsubsubsection}[2]{\hspace{32pt}\oldtocsubsubsection{#1}{#2}}

\usepackage{amsmath}
\usepackage{amssymb}
\usepackage{algorithm}
\usepackage[noend]{algpseudocode}
\makeatletter
\def\BState{\State\hskip-\ALG@thistlm}
\makeatother

\newcommand{\cont}{\supseteq}

\newcommand{\N}{{\mathbb N}}
\newcommand{\Z}{{\mathbb Z}}
\newcommand{\R}{{\mathbb R}}
\newcommand{\C}{{\mathbb C}}
\newcommand{\F}{{\mathbb F}}

\newcommand{\Q}{{\mathbb Q}}

\renewcommand{\cont}[1]{\mathrm{cont}\left( #1 \right)}
\newcommand{\pp}[1]{\mathrm{pp}\left( #1 \right)}
\newcommand{\lc}[1]{\mathrm{lc}\left( #1 \right)}
\newcommand{\sgn}[1]{\mathrm{sgn} \, #1 }

\newcommand{\res}[1]{\mathrm{res}\big( #1 \big)}

\renewcommand{\gcd}{{\mathbf{gcd}}}

\newtheorem{Theorem}{\sc Theorem}[section]
\newtheorem{Lemma}[Theorem]{\sc Lemma}
\newtheorem{Corollary}[Theorem]{\sc Corollary}  
 
\theoremstyle{definition}

\newtheorem{Example}[Theorem]{\sc Example}
\theoremstyle{remark}
\newtheorem{Remark}[Theorem]{\sc Remark}
\newtheorem{Algorithm}{\sc Algorithm}[section]

\begin{document}



\title[On degrees of modular common divisors]{On degrees of modular common divisors and \\ the Big prime gcd algorithm}
\author[Vahagn H.~Mikaelian]{{ Vahagn H.~Mikaelian \\ E\lowercase{-mail: v.mikaelian@gmail.com} }}
\thanks{The author was supported in part by SCS RA, joint Armenian-Russian research project 13RF-030 and by State Committee Science MES RA grant in frame of project 13-1A246.}
\date{\today}

\begin{abstract}
We consider a few modifications of the Big prime modular $\gcd$ algorithm for polynomials in $\Z[x]$. Our modifications are based on bounds of degrees of modular common divisors of polynomials, on estimates of the number of prime divisors of a resultant and on finding preliminary bounds on degrees of common divisors using auxiliary primes. These modifications are used to suggest improved algorithms for  $\gcd$ calculation and for coprime polynomials detection. To illustrate the ideas we apply the constructed algorithms on certain polynomials, in particular, on polynomials from Knuth's example of intermediate expression swell. 
\end{abstract}

\maketitle

{\footnotesize
\parskip0.5mm
\tableofcontents
}

\section{Introduction}
\label{Introduction}

\noindent
This work is one of the articles in which we would like to present parts from new Introduction to computer algebra~\cite{ImGirk}, that currently is under preparation. 
In \cite{ImGirk} we try to give a ``more algebraic'' and detailed view on some of the areas of computer algebra, such as, algorithms on Euclidean rings, extensions of fields, operators in spaces on finite fields, factorization in UFD's, etc..


The Big prime modular $\gcd$ algorithm is one of the first and most popular algorithms of computer algebra. In its classical form it allows to calculate the greatest common divisor $\gcd\big(f(x), g(x)\big)$ for any non-zero polynomials $f(x), g(x) \in \Z [x]$. There are a few modifications of this algorithm for other UFD's, such as multivariate polynomial rings. Attention to the $\gcd$ calculation is partially explained by the first examples that were built to explain importance of application of algebraic methods to computer science. In particular, Knuth's well known example of intermediate expression swell discusses the polynomials
\begin{equation}
\label{Knuth}
\begin{array}{ll} 
\,\,\, f(x) &\!\!\!  = x^8+x^6-3x^4-3x^3+8x^2+2x-5, \\ 
\,\,\, g(x) &\!\!\!  = 3x^6+5x^4-4x^2-9x+21, 
\end{array}
\end{equation}
and it shows that calculation of $\gcd\big(f(x), g(x)\big)$ by traditional Euclidean algorithm on rational numbers generates very large integers to deal with, whereas consideration of these polynomials modulo $p$, that is, consideration of their images under ring homomorphism $\varphi_p : \Z[x] \to \Z_p[x]$ (where $\Z_p [x]$ is the polynomial ring over the residue ring $\Z_p \cong \Z / p\Z$) very easily shows that $\gcd\big(f(x), g(x)\big)=1$ (see~\cite{Knuth 2} and also \cite{Brown, MCA, Davenport S T, Pankratev}). We are going to use the polynomials (\ref{Knuth}) as examples below to apply the algorithms below (see examples~\ref{apply BPA on Knuth}, \ref{apply bound to Knuth}, \ref{enough is 31}, \ref{Knuth for p=2}, \ref{apply modified on Knuth},  \ref{apply plus on Knuth}).

The main idea of the Big prime modular $\gcd$ algorithm is that for the given polynomials  $f(x), g(x) \in \Z [x]$ one may first consider their images $f_p(x)=\varphi_p(f(x)), g_p(x)=\varphi_p(g(x)) \in \Z_p [x]$ under  $\varphi_p$. Unlike $\Z [x]$, the ring $\Z_p [x]$ is an Euclidean domain, since it is a polynomial ring over a field, so the  $\gcd \big(f_p(x), g_p(x) \big)$ can be computed in it by the well known Euclidean algorithm. There remains ``to lift'' a certain fold $t \cdot \gcd \big(f_p(x), g_p(x) \big)$ of it to the ring $\Z [x]$ to reconstruct the pre-image $\gcd \big(f(x), g(x) \big)$. The ``lifting'' procedure consists of selecting the suitable value for prime $p$, then finding in $\Z [x]$ an appropriate pre-image for $\gcd \big(f_p(x), g_p(x) \big)$, then checking if that pre-image divides both $f(x)$ and $g(x)$. If yes, it is the $\gcd \big(f(x), g(x)\big)$ we are looking for. If not, then a new $p$ need be selected to repeat the process. Arguments based on resultants and on Landau-Mignotte bounds show that we can effectively choose $p$ such that the number of required repetitions is ``small''.

The first aim of this work is to present in Sections~\ref{degrees}--\ref{The big prime algorithm} a slightly modified argumentation of the algorithm, based on comparison of the degrees of common divisors of $f(x)$ and $g(x)$ in $\Z [x]$, and of $f_p(x)$ and $g_p(x)$ in $\Z_p [x]$ (see Algorithm~\ref{BPN}). This approach allows some simplification of a step of the algorithm: for some primes $p$ we need not reconstruct the pre-image of $t \cdot \gcd \big(f_p(x), g_p(x) \big)$, but we immediately get an indication that this prime is not suitable, and we should proceed to a new $p$ (see Remark~\ref{my remark}).

Then in Section~\ref{estimating the resultant} we discuss the problem if or not the Big prime modular $\gcd$ algorithm could output the correct answer using just one prime $p$. The answer is positive, but for some reasons it should not be used to improve the algorithm (to make it work with one $p$) because it evolves a too large prime (see Remark~\ref{why not to chose}). Instead, we show that we can estimate the maximal number of $p$'s (repetitions of steps) that may be used in traditional Big prime modular $\gcd$ algorithm. For example, for the polynomials (\ref{Knuth}) of Knuth's example this number is at most 31. Estimates of this type can be found in literature elsewhere. We just make the bound considerably smaller (see Remark~\ref{Let us stress}).

The obtained bounds on the number of primes $p$ are especially effective when we are interested not in $\gcd$, but just in detection if the polynomials $f(x), g(x) \in \Z [x]$ are coprime or not. We consider this in Section~\ref{coprime polynomials} (see Algorithm~\ref{coprime algoritm}).

In Section~\ref{other algorithms} we consider four other ideas to modify the Big prime modular $\gcd$ algorithm. Two first ideas are based on checking the number of primes $p$. The third idea is based on using an auxiliary prime $q$ to estimate the degree of $\gcd \big(f(x), g(x)\big)$ by means of the degree of  $\gcd \big(f_q(x), g_q(x) \big)$ (see Algorithm~\ref{BPN-e}). Example~\ref{apply modified on Knuth} shows how much better results we may get by this modification. The fourth idea combines both approaches: it uses a set of auxiliary primes $q_1, \ldots, q_{k+1}$ to correctly find the degree of $\gcd \big(f(x), g(x)\big)$, and then we use a modified version of Landau-Mignotte bound to find a single big prime $p$ by which we can calculate the $\gcd \big(f_p(x), g_p(x) \big)$.

The arguments used here can be generalized for the case of polynomials on general UFD's. From the unique factorization in a UFD it easily follows, that $\gcd$ always exists, and it is easy to detect if or not the given common divisor of maximal degree is a $\gcd$ or not. The less simple part is to find ways to compute $\gcd$ (without having the prime-power factorization). That can be done for some classes of UFD's, such as, multivariate polynomials on fields.
%
%
The case of general UFD's will be considered later~\cite{General UFD's}.

\section{The $\gcd$ in polynomial rings and the degrees of common divisors}
\label{degrees}

The problem of finding the greatest common divisor $\gcd(a, b)$ of any non-zero elements $a$, $b$  in a ring $R$ can be separated to two tasks: 
\begin{enumerate}
\item finding out if $\gcd(a, b)$, in general, {\it  exists} for $a, b \in R$; and then:
\item finding an {\it effective} way to calculate the $\gcd(a, b)$.
\end{enumerate}
  
The Euclidean algorithm gives an easy answer to both of these tasks in any Euclidean domain, that is, an integrity domain $R$ possessing {\it Euclidean norm} 
$
\delta: R \backslash {0}\to \N \cup \{0\},  
$
such that $\delta(ab)\ge \delta(a)$ hold for any non-zero elements $a, b \in R$; and for any  $a, b \in R$, where $b \not= 0$, there exists elements $q, r \in R$, such that $a=qb+r$, where either $r=0$, or $r\not=0$ and $\delta(r) < \delta(b)$ \cite{Lang, Garrett, Introduction to ring theory, Kostrikin, MCA, Davenport S T}.
The Euclidean algorithm works for any polynomial ring $K[x]$ over a field $K$, such as $\Q[x]$, $\R[x]$, $\C[x]$, $\Z_p[x]$  because these rings can easily be turned to an Euclidean domain by defining $\delta\big(f(x)\big) = \deg f(x)$ for any non-zero $f(x) \in K[x]$.

The situation is less simple in non-Euclidean domains, even in such a widely used ring as the ring $\Z [x]$ of polynomials with integer coefficients. That $\Z [x]$ is not an Euclidean domain is easy to show by elements $x, 2 \in \Z [x]$. If $\Z [x]$ were an Euclidean domain, it would contain elements $u(x), v(x)$ such that $x \cdot u(x) + 2 \cdot v(x) = \gcd(x,2)=\pm 1$, which is not possible.

The first of two tasks mentioned above, namely, {\it existence} of $\gcd$ can be accomplished for $\Z [x]$ by proving that $\Z [x]$ is a UFD, that is, an integrity domain in which every non-zero element $a$ has a factorization $a = \epsilon \,\, p_1 \cdots p_k$, where $\epsilon  \in R^*$ is a unit (invertible) element in $R$,   the elements $p_i$  are prime for all $i=1, \ldots , k$, and where the factorization above is unique in the sense that if $a$ has another factorization of that type $\theta \,\, q_1 \cdots q_s$, where $\theta  \in R^*$ and the elements $q_i$   are prime, then $k=s$ and (perhaps after some reordering of the prime factors) the respective prime elements are associated: $p_i \approx q_i$ for all $i=1, \ldots , k$. For briefness, in the sequel we will often omit the phrase ``perhaps after some reordering of the prime factors'' and this will cause no confusion.

After merging  the associated prime elements together, we get a unique factorization into prime-power elements:
\begin{equation}
\label{factorization a}
\text{ $a = \nu \,\, p^{\alpha_1}_1 \cdots p^{\alpha_n}_n$, \,\, $\nu  \in R^*$,\,\, $\alpha_i \in \N$ and $p_i \not\approx p_j$ for any $i \not= j$;\,\, $i,j = 1, \ldots, n$}
\end{equation}
(in some arguments below we may admit some of the factors $p^{\alpha_i}_i$ participate with degrees $\alpha_i=0$, this makes some notations simpler). From this it is easy to see that in a UFD $R$ the $\gcd(a,b)$ exists for any non-zero elements $a, b \in R$. 
Assume $b \in R$ has the factorization 
$$
b = \kappa \,\, p^{\alpha'_1}_1 \cdots p^{\alpha'_n}_n, \,\, \kappa \in R^*
$$ 
(we use the same primes $p_i$ in both factorizations because if, say, $p_i$ is not actually participating in one of those factorizations, we can add it as $p^{\alpha_i}_i$ with $\alpha_i = 0$). Then 
\begin{equation}
\label{min powers}
\gcd(a, b)= d = p^{\gamma_1}_1 \cdots p^{\gamma_n}_n,
\end{equation}
where $\gamma_i=\min \{\alpha_i , \alpha'_i \}$. This follows from uniqueness of factorization in UFD. For, if $h$ is a common divisor of $a, b$, and if $p_i$ is a prime divisor of $h$, then it also is a prime divisor of $a$ and of $b$.
The elements $p_i$ cannot participate in factorization of $h$ by a power greater than $\min \{\alpha_i , \alpha'_i \}$, because then $a$ (or $b$) would have an alternative factorization in which $p_i$ occurs more than $\alpha_i$ (or $\alpha'_i$) times.

The shortest way to see that $\Z[x]$ is a UFD is to apply Gauss's Theorem: if the ring $R$ is a UFD, then the polynomial ring $R[x]$ also is a UFD~\cite{Lang, Garrett, Basic algebra, Kostrikin, MCA}. Since $\Z$ is a UFD (that fact is known as ``the fundamental theorem of arithmetic''), $\Z[x]$ also is a UFD.

Clearly, $\gcd(a, b)$ is defined up to a unit multiplier from $R^*$. For integers from $R=\Z$ or for polynomials from $R=\Z[x]$ this unit multiplier can be just $-1$ or $1$. So to say: $\gcd(a, b)$ is defined ``up to the sign $\pm 1$'' because $\Z^*=\Z[x]^*=\{-1, 1\}$. And for polynomials from $R=\Z_p[x]$ the $\gcd(a, b)$ is defined up to any non-zero multiplier $t \in \Z_p^* = \{1,\ldots, p-1\}$. Taking this into account we can use $\gcd(a, b)=1$ and $\gcd(a, b) \approx 1$ as equivalent notations, since associated elements are defined up to a unit multiplier. Notice that in some sources they prefer to additionally introduce a normal form of the $\gcd$ to distinguish one fixed instance of the $\gcd$. Instead of using that extra term, we will just in a few places refer to the ``positive $\gcd$'', meaning that we take, say,  $2 = \gcd(6, 8)$, and not $-2$.

Furthermore, since the content $\cont{f(x)}$ of a polynomial $f(x)$ is a $\gcd$ for some elements (coefficients of the polynomials), the constant and the primitive part $\pp{f(x)}=f(x)/\cont{f(x)}$ can also be considered up to a unit multiplier.
For a non-zero polynomial $f(x) \in \Z[x]$ we can choose the $\cont{f(x)}$ so that  $\sgn{\cont{f(x)}} = \sgn{\lc{f(x)}}$, that is, the $\cont{f(x)}$ has the same sign as the leading coefficient of  $f(x)$. Then the leading coefficient $\lc{\pp{f(x)}}$ of the primitive part $\pp{f(x)}=f(x)/\cont{f(x)}$ will be positive. We will use this below without special notification.

Now we would like to a little restrict the algebraic background we use. Two main algebraic systems, used in the Big prime modular $\gcd$ algorithm, are the Euclidean domains and the UFD's. However, their usage is ``asymetric'' in the sense that the Euclidean domains and Euclidean algorithm are used in many parts of the Big prime modular $\gcd$ algorithm, whereas the UFD's are used just to prove that $\gcd$ does exist. Moreover, it is easy to understand that (\ref{factorization a}) and (\ref{min powers}) may hardly be effective tools to calculate a $\gcd$, since they are using factorization of elements to primes, while finding such a factorization is a more complicated task than finding just the $\gcd$. Thus, it is reasonable to drop the UFD's from consideration, and to obtain (\ref{factorization a}) directly using Gauss's Lemma on primitive polynomials in $\Z[x]$ (a polynomial $f(x) \in \Z[x]$ is primitive if $\cont{f(x)}\approx 1$, that is, $\pp{f(x)}= f(x) / \cont{f(x)}\approx f(x)$).

By Gauss's Lemma, a product of two primitive polynomials is primitive in $\Z[x]$~\cite{Lang, Garrett, Basic algebra, Kostrikin, MCA}. So if 
\begin{equation}
\label{decompositions}
\text{$f(x)= \cont{f(x)} \cdot \pp{f(x)}$ and $g(x)= \cont{g(x)} \cdot \pp{g(x)}$},
\end{equation}
then 
\begin{equation}
\label{two decompositions}
\begin{array}{l} 
\cont{f(x) \cdot g(x)}  = \cont{f(x)}\cdot \cont{g(x)} \\ 
\pp{f(x)\cdot g(x)}  = \pp{f(x)}\cdot \pp{g(x)}.
\end{array}
\end{equation}

The following is easy to deduce from Gauss's Lemma
\begin{Lemma}
\label{divides in Z and Q}
If $f(x), t(x) \in \Z[x]$ and $t(x)$ is primitive, then if $t(x)$ divides $f(x)$ in the ring $\Q[x]$, then $t(x)$ also divides $f(x)$ in  $\Z[x]$.
\end{Lemma}

The unique factorization of any non-zero $f(x) \in \Z[x]$ is easy to obtain from decompositions (\ref{two decompositions}) above and from Lemma~\ref{divides in Z and Q}. Let us just outline it,  the details can be found in~\cite{Lang, Garrett, MCA, Davenport S T, ImGirk}. 
By the fundamental theorem of arithmetic $\cont{f(x)}$ can in a unique way be presented as a products of powers of primes:
$
\cont{f(x)} = \nu \,\, p^{\alpha_1}_1 \cdots p^{\alpha_n}_n.
$
So, if $\deg f(x) = 0$, then we are done.

Assume $\deg f(x)>0$. If $f(x)$ is not prime, then by repeatedly splitting it to products of factors of lower degree as many times as needed, we will eventually get a presentation of $f(x)$ as a product of  $\cont{f(x)}$ and of some finitely many primitive prime polynomials $q_i(x)$ of degrees greater than $0$. We do not yet have the uniqueness of this decomposition, but we can  still group the associated elements together to get the presentation:
\begin{equation}
\label{decom. first}
f(x)= \cont{f(x)} \cdot \pp{f(x)}= \nu \,\, p^{\alpha_1}_1 \cdots p^{\alpha_n}_n \, \cdot \, q^{\beta_1}_1(x) \cdots q^{\beta_m}_m(x).
\end{equation}
If $f(x)$ has another, alternative presentation of this sort and if $t(x)$ is one of the primitive prime factors (of  degree greater than $0$) of that presentation, then the product (\ref{decom. first}) is divisible by $t(x)$. By Lemma \ref{divides in Z and Q}, $t(x)$ divides $f(x)$ also in $\Z[x]$. Since $t(x)$ is prime, it is associated to one of $q_i(x)$. Eliminate one instance of this $q_i(x)$ in (\ref{decom. first}) and consider $f(x)/q_i(x)$. If $f(x)/q_i(x)$ also is divisible by $q_i(x)$, we repeat the process. If not, we turn to other primitive prime polynomials (of degree greater than $0$) dividing what remains from (\ref{decom. first}) after eliminations. After finitely many steps (\ref{decom. first}) will become  $\nu \,\, p^{\alpha_1}_1 \cdots p^{\alpha_n}_n$, and also from the other, alternative presentation a constant should be left only. So we apply  the fundamental theorem of arithmetic  one more time to get that (\ref{decom. first}) is the unique factorization.

We see that (\ref{decom. first}) is a particular case of (\ref{factorization a}). The proof above avoided usage of Gauss's Theorem and the formal definitions of the UFD's. And we see that the prime elements of $\Z[x]$ are of two types: {\it prime numbers} and  {\it primitive prime polynomials} of degrees greater than $0$.

Existence of $\gcd\big(f(x), g(x)\big)$ for any two non-zero polynomials in $f(x), g(x) \in \Z[x]$ can be deduced from (\ref{decom. first}) in analogy with (\ref{min powers}). If 
\begin{equation}
\label{decom. second}
g(x) = \nu' \,\, p^{\alpha'_1}_1 \cdots p^{\alpha'_n}_n \cdot q^{\beta'_1}_1(x) \cdots q^{\beta'_m}_m(x),
\end{equation}
then 
\begin{equation}
\label{min powers Z}
\gcd\big(f(x), g(x)\big) = \kappa \,\, p^{\gamma_1}_1 \cdots p^{\gamma_n}_n \cdot q^{\delta_1}_1(x) \cdots q^{\delta_m}_m(x),
\end{equation}
where $\kappa= \pm 1$, \,  $\gamma_i=\min \{\alpha_i , \alpha'_i \}$, \,  $\delta_j=\min \{\beta_j , \beta'_j \}$; \, ($i=1,\ldots,n; \,\, j=1,\ldots,m$).
However, like we admitted earlier, (\ref{min powers}) and (\ref{min powers Z}) are no effective tools to calculate the $\gcd$. We will turn to $\gcd$ calculation algorithm in the next section.

\vskip4mm
(\ref{min powers}) and (\ref{min powers Z}) allow us to get some information that we will be essential later. Observe that the following definition of $\gcd$, often used in elementary mathematics, is no longer true for general polynomial rings: {\it ``$d(x)$ is the greatest common divisor of $f(x)$ and $g(x)$ if it is their common divisor of maximal degree''}. For example, for $f(x)=12x^2+24x+12$ and $g(x)=8x+8$ the maximum of degrees of their common divisors is $1$. Nevertheless, $h(x) = x + 1$ is not the $\gcd\big(f(x), g(x)\big)$, although $h(x) | f(x)$, $h(x) | g(x)$ and $\deg h(x)=1$. For, $x+1$ is not divisible by the common divisor $2x+2$. We can detect the cases when the divisor of highest degree is the $\gcd$.

\begin{Lemma}
\label{highest degree is gcd}
For polynomials $f(x), g(x) \in \Z[x]$ their common divisor of maximal degree $h(x)$  is their $\gcd$ if and only if $\cont{h(x)} \approx \gcd \big(\cont{f(x)}, \cont{g(x)}\! \big)$.
\end{Lemma}

The lemma easily follows from (\ref{decom. first}), (\ref{decom. second}) and (\ref{min powers Z}). We see that in example above the condition was missing: $\cont{x+1}=1$ but $\gcd \big(\cont{f(x)}, \cont{g(x)}\! \big) = \gcd (12,8) \approx 4 \not\approx 1$. In fact, $\gcd \big(f(x), g(x)\big) \approx 4x + 4$.

\begin{Corollary}
\label{highest degree is gcd for primitives}
For primitive polynomials $f(x), g(x) \in \Z[x]$ their common divisor of maximal degree $h(x)$ is their $\gcd$ if and only if $h(x)$ is primitive.
\end{Corollary}

In the case if polynomials are over a field, the situation is simpler. For any field $K$ the polynomial ring $K[x]$ is a UFD (and even an Euclidean domain). Any non-zero $f(x) \in K[x]$ has a factorization 
\begin{equation}
\label{factorization on field}
f(x)= \theta \,\, \cdot \, q^{\beta_1}_1(x) \cdots q^{\beta_m}_m(x), \,\,\, \theta \in K^*, \,\,\, \deg{q_i(x)}>0, \,\,\, i=1,\ldots,m,
\end{equation}
which is unique in the sense mentioned above. Since all non-zero scalars in $K$ are units, what we in (\ref{decom. first}) above had  as a product of some prime numbers, actually ``merges'' in $K$ into a unit: 
$$ 
\nu \,\, \cdot p^{\alpha_1}_1 \cdots p^{\alpha_n}_n = \theta \in K^* = K \backslash \{0\}.
$$

Comparing factorizations of type (\ref{factorization on field}) for any non-zero polynomials $f(x), g(x) \in K[x]$ we easily get: 
\begin{Lemma}
\label{highest degree is gcd for fields}
For any non-zero polynomials $f(x), g(x) \in K[x]$ over a field $K$ their common divisor of maximal degree $h(x)$ is their $\gcd$.
\end{Lemma}

This, in particular, is true for rings mentioned above: $\Q[x]$, $\R[x]$, $\C[x]$, $\Z_p[x]$. We will use this fact later to construct the Big prime modular $\gcd$ algorithm and its modifications. 

The analog of Lemma~\ref{highest degree is gcd for fields} was not true for $\Z[x]$ because in factorization  (\ref{min powers Z}) we have the non-unit prime-power factors $p^{\gamma_i}_i $ which do participate in factorization of $d(x)= \gcd\big(f(x), g(x)\big)$, but which add {\it nothing} to the degree of $d(x)$. This is why maximality of the degree is no longer the only criterion in $\Z[x]$ to detect if the given $h(x)$ is the $\gcd$ or not.

\section{Some notations for modular reductions}
\label{reduction notations}

The following notations, adopted from~\cite{ImGirk}, are to make our arguments shorter and more uniform when we deal with numerals, polynomials and matrices. As above, let $\Z_p$ be the residue ring (finite Galois field $\Z_p = \F_p= \{0, \ldots, p-1\}$) and 
$$
\varphi_p: \Z \to \Z_p
$$
be the rings homomorphism mapping each $z \in \Z$ to the reminder after division of $z$ by $p$. That is, 
$
\varphi_p(z)\equiv z ({\rm mod} \,\, p)
$,
and 
$
\varphi_p(z) \in \Z_p
$.

We use the same symbol $\varphi_p$ to denote the homomorphism 
$$
\varphi_p: \Z[x] \to \Z_p[x],
$$
where $\Z_p[x]$ is the ring of polynomials over $\Z_p$, and $\varphi_p$ is mapping each of the coefficients $a_i$ of $f(x)\in \Z[x]$ to the reminder after division of $a_i$ by $p$. 

Similarly, we define the homomorphism of matrix rings 
$$
\varphi_p: M_{m,n}(\Z) \to M_{m,n}(\Z_p), 
$$
which maps each of the elements $a_{ij}$ of a matrix $A \in M_{m,n}(\Z)$ to  the reminder after division of $a_{ij}$ by $p$. 

Using the same symbol $\varphi_p$ for numeric, polynomial and matrix homomorphisms causes no misunderstanding below, and it is more comfortable for some reasons. These homomorphisms are called  {\it ``modular reductions''} or just {\it ``reductions''}. We can also specify these homomorphism as: {\it ``numeric modular reduction''}, {\it ``polynomial modular reduction''} or {\it ``matrix modular reduction''} where needed~\cite{ImGirk}. 

For $a \in \Z$ denote $\varphi_p(a) = a_p$. For $f(x) \in \Z[x]$ denote $\varphi_p\big(f(x)\big) = f_p(x)\in \Z_p[x]$. So if
\begin{equation}
\label{f(x)}
f(x) = a_0 x^n + \cdots + a_n
\end{equation}
then: 
$$
f_p(x) = \varphi_p\big(f(x)\big) = \varphi_p(a_0) x^n + \cdots + \varphi_p(a_n) 
= a_{0,p} x^n + \cdots + a_{n,p} \in \Z_p[x].
$$
And for a matrix $A \in M_{m,n}(\Z)$ denote $\varphi_p(A) = A_p \in M_{m,n}(\Z_p)$. If $A = \lVert a_{i,j} \rVert_{m \times n}$ then $A_p = \lVert \varphi_p(a_{i,j}) \rVert_{m \times n} = \lVert a_{i,j,p} \rVert_{m \times n}$.

\section{Problems at lifting the modular $\gcd$ to $\Z[x]$}
\label{four problems}

Now we turn to the second task mentioned earlier: {\it effective calculation} of the actual $\gcd \big(f(x), g(x)\big)$ for the given non-zero polynomials $f(x), g(x) \in \Z[x]$. 

The ring $\Z_p[x]$ is an Euclidean domain, unlike the ring $\Z[x]$. So we can use the Euclidean algorithm to calculate $\gcd$ for any non-zero polynomials in $\Z_p[x]$, including the modular images $f_p(x)$ and $g_p(x)$. 
Since the notation $\gcd\big(f_p(x), g_p(x)\big)$ is going to be used repeatedly, for briefness denote by $e_p(x)$ the $\gcd$ calculated by Euclidean algorithm for $f_p(x), g_p(x)$. Let us stress that $\gcd\big(f_p(x), g_p(x)\big)$ is not determined uniquely, since for any non-zero $t \in \Z_p$ the product $t \cdot \gcd\big(f_p(x), g_p(x)\big)$ also is a $\gcd$ for $f_p(x), g_p(x)$. We are denoting just {\it one} of these $\gcd$'s (namely, that computed by the Euclidean algorithm) by $e_p(x)$. This $e_p(x)$ is unique, since at each step of the Euclidean algorithm we have a unique action to take (to see this just consider the steps of ``long division'' used to divide $f_p(x)$ by $g_p(x)$ on field $\Z_p$). 

The main idea of the  algorithm is to calculate the $e_p(x) \approx \gcd\big(f_p(x), g_p(x)\big)$ for some suitable $p$, and to reconstruct $d(x) = \gcd\big(f(x), g(x)\big)$ by it.
We separate the process to four main problems that may occur, and show how to overcome each one to arrive to a correctly working algorithm.

\vskip2mm
\subsection{Problem 1. Avoiding the eliminating coefficients} 
$\phantom .$

After reduction $\varphi_p$ some of the coefficients of $f(x)$ and $g(x)$ may change or even eliminate. So their images $f_p(x)=\varphi_p\big(f(x)\big)$ and $g_p(x)=\varphi_p\big(g(x)\big)$ may keep very little information to reconstruct the $d(x)$ based on $e_p(x)$. 

\begin{Example}
If $f(x)=7x^2+22$ and $g(x)=49x^3+154x$ then for $p=7$ we get $f_p(x) = 1$ and $g_p(x) = 0$. So these values contain no reliable information to reconstruct the $\gcd\big(f(x), g(x)\big)$.
\end{Example}

The first simple idea to avoid such eliminations is to take $p$ larger than the absolute value of all coefficients of $f(x)$ and $g(x)$. This, however, is not enough since a divisor $h(x)$ of a polynomial $f(x)$ may have coefficients, larger than those of $f(x)$. Moreover, using the cyclotomic polynomials for large enough $n$:
$$
\phi_n (x)= \hskip-6mm \prod_{\,\,\,\, k=1 , \cdots , n; \,(k,n)=1} \hskip-9mm \big(x-e^{\frac{2i\pi k}{n}}\big)
$$ 
one can get divisors of $f(x)=x^n-1$ which have a coefficient larger than any pre-given number~\cite{Garrett, MCA, ImGirk}. Since we do not know the divisors of $f(x)$ and $g(x)$, we cannot be sure if the above mentioned large $p$ will be large enough to prevent eliminations of coefficients of $h(x)$. To overcome this one can use the Landau-Mignotte bounds\footnote{%
In different sources the bounds on coefficients of the divisors are called differently, associating them by names of L.~Landau or M.~Mignotte or by both of them. These authors have different roles in development of the formulas, which in turn are consequence of a formula by A. L. Cauchy.}, as it is done in~\cite{Davenport S T, MCA, Pankratev}.  For a polynomial $f(x)$ given by (\ref{f(x)}) denote its norm by $\lVert f(x) \rVert = \sqrt{\sum^n_{i=0 }a_i^2}$.

\begin{Theorem}[L.~Landau, M.~Mignotte]
\label{Landau Mignotte} 
Let $f(x) = a_0 x^n + \cdots + a_n$ and $h(x) = c_0 x^k + \cdots + c_k$ be non-zero polynomials in $\Z[x]$. If $h(x)$ is a divisor of $f(x)$, then:
\begin{equation}
\label{L-M}
\sum^n_{i=0 }| c_i | \le 2^k  \cdot |c_0/a_0 | \cdot \lVert f(x) \rVert.
\end{equation}
\end{Theorem}
The proof is based on calculations on complex numbers, and it can be found, for example, in~\cite{MCA, ImGirk}. We are going to use the Landau-Mignotte bounds in the following two shapes:

\begin{Corollary}
\label{N_f corollary}
In notations of Theorem~\ref{Landau Mignotte} there is the following upper bound for the coefficients of $h(x)$:
\begin{equation}
\label{N_f}
| c_i | \le N_f =  2^{n-1} \lVert f(x) \rVert.
\end{equation}
\end{Corollary}
\begin{proof}
To obtain this from (\ref{L-M}) first notice that $|c_0/a_0 | \le 1$. 

Next, if $k = \deg h(x) = \deg f(x) = n$, then $f(x)=r \cdot h(x)$, where $r$ is a non-zero integer. Then $| c_i | \le \max\{| c_i | \, | \, i=0,\ldots, n\} \le \max\{| a_i | \, | \, i=0,\ldots, n\}\le \lVert f(x) \rVert$. 

Finally, if $k =\deg h(x) \le  n-1$ ($k$ is unknown to us), then we can simply replace in (\ref{L-M}) the value $2^k$ by $2^{n-1}$. 
\end{proof}

\begin{Remark}
\label{smaller degree}
In literature they use the rather less accurate bound $| c_i | \le 2^{n} \lVert f(x) \rVert$, but the second paragraph of our proof above allows to replace $2^{n}$ by $2^{n-1}$\!. See also Remark~\ref{Let us stress} .
\end{Remark}

\begin{Corollary}
\label{N_fg corollary}
In notations of Theorem~\ref{Landau Mignotte}, if $h(x)$ also is a divisor of the polynomial $g(x)=b_0 x^m+\cdots+b_m$, then  there is the following upper bound for the coefficients of $h(x)$:
\begin{equation}
\label{N_fg}
| c_i | \le N_{f,g} = 2^{\min \{ n,m \}} \cdot \gcd(a_0, b_0) \cdot
\min
\left\{
{{\lVert f(x) \rVert}\over{|a_0|}},
{{\lVert g(x) \rVert}\over{|b_0|}}
\right\}.
\end{equation}
\end{Corollary}
\begin{proof}
To obtain this from (\ref{L-M}) just notice that if $h(x)$ is a common divisor for $f(x)$ and $g(x)$, then its leading coefficient $c_0$ divides both $a_0$ and $b_0$. 
\end{proof}

Formula (\ref{N_fg}) provides the hint to overcome Problem 1 about eliminating coefficients, mentioned at the start of this subsection. Although the divisors $h(x)$ of $f(x)$ and $g(x)$ are yet unknown, we can compute $N_{f,g}$ and take $p > N_{f,g}$. If we apply the reduction $\varphi_p$ for this $p$ we can be sure that none of the coefficients of $h(x)$ has changed ``much'' under that homomorphism, for, $\varphi_p$ does not alter the non-negative coefficients of $h(x)$, and it just adds $p$ to all negative coefficients of $h(x)$. The same holds true for  $d(x) = \gcd \big(f(x), g(x)\big)$.

\vskip2mm
\subsection{Problem 2. Negative coefficients and reconstruction of the pre-image} 
$\phantom .$\vskip-4mm

The reduction $\varphi_p$ is not a bijection, and $d_p(x)$ has infinitely many pre-images in $\Z[x]$. But the relatively uncomplicated relationship between coefficients of $d(x)$ and $d_p(x)$, obtained in previous subsection, may allow us to reconstruct $d(x)$ if we know $d_p(x)$. 
The condition $p > N_{f,g}$ puts a restriction on the pre-image $d(x)$: the coefficients of $d(x)$ are either equal to respective coefficients of $d_p(x)$ (if they are non-negative), or they are the respective  coefficients of $d_p(x)$ minus $p$ (if they are negative). Reconstruction may cause problems connected with negative coefficients.

\begin{Example}
If for some polynomials $f(x), g(x)$ we have $N_{f,g} = 15$, we can take the prime, say, $p = 17 > N_{f,g}$. Assume we have somehow calculated $d_{17}(x)=12x^3+3x+10$, we can be sure that $d(x)$ is not the pre-image $29x^3- 17x^2+ 20x+27$ because $d(x)$ cannot have coefficients greater than $15$ by absolute value. But we still cannot be sure if the pre-image $d(x)$ is $12x^3+3x+10$, or $-5x^3+3x+10$, or maybe $-5x^3-14x-7$. 
\end{Example}

It is easy to overcome this by just taking a larger value: 
$$
p > 2 \cdot N_{f,g}.
$$
If the coefficient $c_i$ of $d(x)$ is non-negative, then $\varphi_p (c_i) = c_i < p/2$, and if it is negative, then $\varphi_p (c_i) = c_i + p > p/2$. This provides us with the following very simple algorithm to reconstruct $d(x)$ if we have already computed $d_p(x)$ for sufficiently large prime $p$.

\begin{Algorithm}[The polynomial reconstruction by modular image]
\label{Alg reconstruct} 
{\it $\phantom{.}$
\\
Input: For an unknown polynomial $d(x) \in \Z[x]$ we know the upper bound $N$ of absolute values of its coefficients, and for arbitrarily large prime number $p$ we have the modular image $d_p(x) = b_{0,p} x^k+ \cdots +b_{k,p} \in \Z_p[x]$.\\
Reconstruct the polynomial $d(x)$.
}

{\parskip1mm

\noindent
01. Choose any prime $p > 2\cdot N$.

\noindent
02. Set $k = \deg{d_p(x)}$.

\noindent
03. Set $i = 0$.

\noindent
04. While $i \le k$

\noindent
05. \hskip5mm
if $b_{i,p} < p/2$

\noindent
06. \hskip13mm
set $b_{i} = b_{i,p}$;

\noindent
07. \hskip5mm
else

\noindent
08. \hskip13mm
set $b_{i} = b_{i,p}-p$;

\noindent
10. \hskip5mm
set $i = i + 1$.

\noindent
11. Output $d (x) = b_{0} x^k+ \cdots +b_{k}$.
}
\end{Algorithm}

\vskip2mm
\subsection{Problem 3. Finding the correct fold of the modular $\gcd$ of right degree} 
$\phantom .$
\vskip-4mm

Now additionally assume the polynomials $f(x), g(x) \in \Z[x]$ to be primitive. Since $\cont{f(x)}$ and $\cont{g(x)}$ are defined up to the sign $\pm 1$,  we can without loss of generality admit the leading coefficients of $f(x), g(x)$ to be positive.  

Below, in Problem 4, we will see that for some $p$ the polynomial $e_p(x)$, computed by the Euclidean algorithm in $\Z_p[x]$, may not be the image of $d(x)$ and, moreover, its degree may be different from that of $d(x)$. This means that by applying Algorithm~\ref{Alg reconstruct} to $e_p(x)$ we may not obtain $d(x)$.  Assume, however, we have a $p$, which meets the condition $p > 2\cdot N_{f,g}$ and for which:
\begin{equation}
\label{degrees equal}
\deg d(x) = \deg e_p(x).
\end{equation}
By Corollary~\ref{highest degree is gcd for primitives} a common divisor of $f(x), g(x)$ is the $d(x)= \gcd\big(f(x), g(x)\big)$ if and only if it is primitive and if its degree is the maximum of degrees of all common divisors. Since $\varphi_p$ does not change the degree of $d(x)$, we get by Lemma~\ref{highest degree is gcd for fields} (applied for the field $K= \Z_p$) that $d_p(x)$ is a $\gcd$ of $f_p(x), g_p(x)$ in $\Z_p[x]$. This correspondence surely is not on-to-one, because in $\Z[x]$ the $\gcd$ is calculated up to the unit element of $\Z[x]$, which is $\pm 1$, whereas in  $\Z_p[x]$ the $\gcd$ is calculated up to the unit element of $\Z_p[x]$, which can be any non-zero number $t\in \Z^*_p=\{1,\ldots, p-1\}$. 
So the polynomial $e_p(x)$ calculated by the Euclidean algorithm may not be the image $d_p(x)$ of $d(x)$. 

\begin{Example}
For $f(x)=x^2+4x+3$ and $g(x)=x^2+2x+1$ whichever prime $p>4$ we take, we will get by the Euclidean algorithm:
$$
e_p(x) = \gcd\big(f_p(x), g_p(x)\big)=2x+2\in \Z_p[x]. 
$$
But in $\Z[x]$ we have $d(x)= x+1$. So regardless how large $p$ we choose, we will never get $\varphi_p(x+1)=2x+2$. 
\end{Example}

In other words, we are aware that the image $d_p(x)$ is one of the folds $t\cdot e_p(x)$ of $e_p(x)$ for some $t \in \{1,\ldots, p-1 \}$, but we are not aware which $t$ is that. 

The leading coefficient $c_0 = \lc{d(x)}$ of $d(x)$ can also be assumed to be positive. Denote by $w$ the positive $\gcd(a_0, b_0)$. Since both $c_0$ and $w$ are not altered by $\varphi_p$, their fraction $w/c_0$ also is not altered. Take such a $t$ that:
\begin{equation}
\label{fold}
\lc{t\cdot e_p(x)} = w.
\end{equation}
Even if $t\cdot e_p(x)$ is not the image $d_p(x)$,  it is the image of $l\cdot d(x)$, where $l$ divides $w/c_0$. If we calculate the pre-image $k(x)\in \Z[x]$ of $t\cdot e_p(x)$ by Algorithm~\ref{Alg reconstruct}, we will get a polynomial, which is either $d(x)$ or is some fold of $d(x)$. Since $f(x), g(x)$ are primitive, it remains to go to the primitive part $d(x) = \pp{k(x)}$.

The general case, when $f(x), g(x)$ may not be primitive, can easily be reduced to this: for arbitrary $f(x), g(x)$
take their decompositions by formula (\ref{decompositions}) and set
\begin{equation}
\label{r}
r=\gcd\big(\cont{f(x)}, \cont{g(x)}\!\big)\in \Z.
\end{equation}
Then assign $f(x)=\pp{f(x)}$, $g(x)=\pp{g(x)}$ and do the steps above for these new polynomials. After the $d(x) = \pp{k(x)}$ is computed, we get the final answer as $r\cdot d(x) = r\cdot \pp{k(x)}$

Notice that for Algorithm~\ref{Alg reconstruct} we need $p$ to be greater than any coefficient $|c_i|$ of the polynomial we reconstruct. The bound $p > 2\cdot N_{f,g}$ assures that $p$ meets this condition for $d(x)$. We, however, reconstruct not $d(x)$ but $l \cdot d(x)$, which may have larger coefficients. One could overcome this point by taking $p > w \cdot 2 \cdot N_{f,g}$ but this is not necessary because, as we see later, while the Big prime modular $\gcd$ algorithm works, the value of $p$ will grow and this issue will be covered.

\vskip2mm
\subsection{Problem 4. Finding the right degree for the modular $\gcd$} 
$\phantom .$

As we saw, one can reconstruct $d(x)$ if we find a $p > 2\cdot N_{f,g}$ such that the condition (\ref{degrees equal}) holds. Consider an example to see that (\ref{degrees equal}) may actually not hold for some $p$ even if $\varphi_p$ is not altering the coefficients of $f(x)$ and $g(x)$! 

\begin{Example}
For $f(x)=x^2+1$ and $g(x)=x+1$ we have $d(x) = \gcd\big(f(x), g(x)\big)=1$. Taking $p=2$ we get $f_2(x)=x^2+1$ and $g_2(x)=x+1$. In $\Z_2[x]$ we have $f_2 (x)=x^2+1=x^2+1^2=(x+1)(x+1)$, thus,  $e_2(x) = \gcd\big(f_2(x), g_2(x)\big)=x+1$. We get that $1 = \deg(x+1) > \deg(x) = 0$. In particular, whatever $t$ we take, $t \cdot (x+1)$ is not the image of $d(x)=1$ under $\varphi_2$.
\end{Example}

The idea to overcome this problem is to show that the number of primes $p$, for which (\ref{degrees equal}) falsifies, is ``small''. So if the selected $p$ is not suitable, we take another $p$ and do the calculation again by the new prime. And we will not have to repeat these steps for many times (we will turn to this point in Section~\ref{estimating the resultant}).

The proof of the following theorem and the definition of the resultant $\res {f(x),g(x)}$ (that is, of the determinant of the Sylvester matrix $S_{f,g}$ of polynomials $f(x),g(x)$) can be found, for example, in~\cite{Lang, MCA, Kostrikin, ImGirk}. The resultant is a comfortable tool to detect if the given polynomials are coprime:

\begin{Theorem}
\label{resultant zero}
Let $R$ be an integrity domain. The polynomials $f(x), g(x) \in R[x]$ are coprime if and only if $\res {f(x),g(x)} \not= 0$.
\end{Theorem}

The following fact in a little different shape can be found in~\cite{MCA} or~\cite{Davenport S T}: 

\begin{Corollary}
\label{resultant fold p}
If the prime $p$ does not divide at least one of the leading coefficients $a_0$, $b_0$ of polynomials, respectively, $f(x), g(x) \in \Z[x]$ then $\deg d(x) \le \deg e_p (x)$.
If $p$ also does not divide $R=\res{f(x)/d(x),g(x)/d(x)}$, where $d(x) = \gcd \big(f(x), g(x)\big)$, then 
\begin{equation}
\label{modular degrees equal}
 \deg d (x)=\deg d_p (x)=\deg \gcd\big(f_p (x),g_p (x)\big)=\deg e_p (x).
\end{equation}
\end{Corollary}

\begin{proof}
Since $c_0 = \lc{d(x)}$ divides $w= \gcd (a_0, b_0)$, then  $\varphi_p(c_0) \not= 0$ by the choice of $p$. Thus, $\deg d (x)=\deg d_p (x) \le \deg \gcd\big(f_p (x),g_p (x)\big)$.

Since  $d_p(x)\not= 0$,  we can consider the fractions $f_p(x)/d_p(x)$ and $g_p(x)/d_p(x)$ in $\Z_p[x]$. From unique factorizations of $f_p (x)$ and $g_p (x)$ in UFD $\Z_p[x]$ it is very easy to deduce that
$$
e_p(x)
\approx \gcd\big(f_p (x),g_p (x)\big) 
\approx d_p (x) \cdot \gcd\big(f_p (x)/d_p (x),g_p (x)/d_p (x)\big).
$$
In particular, $\deg d(x) = \deg d_p(x) \le \deg e_p (x)$.
And the inequality $\deg d_p (x) \not= \deg e_p(x)$ may occur only if 
$$
\deg \gcd\big(f_p (x)/d_p (x),g_p (x)/d_p (x)\big) >0,
$$
that is, when $f_p(x)/d_p(x)$ and $g_p(x)/d_p(x)$ are not coprime in $\Z_p[x]$ or, by Theorem~\ref{resultant zero}, when $\res{f_p (x)/d_p (x),g_p (x)/d_p (x))}=0$. The latter is the determinant of Sylvester matrix $S_{f_p/d_p, \, g_p/d_p}$. 
Consider the matrix rings homomorphism (matrix modular reduction)
$$
\varphi_p: M_{m+n}(\Z) \to M_{m+n}(\Z_p),
$$
where $n = \deg f (x)$, $m = \deg g (x)$ (as mentioned earlier we use the same symbol $\varphi_p$ for numeric, polynomial and matrix reductions). Since, $\varphi_p (S_{f/d, \, g/d}) = S_{f_p/d_p, \, g_p/d_p}$, and since the determinant of a matrix is a sum of products of its elements, we get
$$
R_p = 
\varphi_p (R) =
\varphi_p \Big(
\res{f(x)/d(x),g(x)/d(x)}
\Big) =
\res{f_p (x)/d_p (x),g_p (x)/d_p (x))}.
$$
So $R_p$ can be zero  if and only if $R$ is divisible by $p$. 
The polynomials $f(x)/d(x)$ and $g(x)/d(x)$ are coprime in $\Z[x]$ and their resultant is not zero by Theorem~\ref{resultant zero}. And $R$ cannot be a positive integer divisible by $p$ since that contradicts the condition of this corollary.
\end{proof}

Corollary~\ref{resultant fold p} shows that if for some $p$ the equality (\ref{degrees equal}) does not hold for  polynomials $f(x), g(x) \in \Z[x]$, then $p$ divides either $a_0$ and $b_0$, or it divides the resultant $R$. We do not know $R$, since we do not yet know $d(x)$ to calculate the resultant $R=\res{f(x)/d(x),g(x)/d(x)}$. But, since the number of such primes $p$ is just finite, we can arrive to the right $p$ after trying the process for a few primes. We will turn to this again in Section~\ref{estimating the resultant}.

\section{The Big prime modular $\gcd$ algorithm}
\label{The big prime algorithm}

Four steps of the previous section provide us with the following procedure. We keep all the notations from Section~\ref{four problems}. Take the primitive polynomials $f(x), g(x) \in \Z[x]$. Without loss of generality we may assume $a_0, b_0 > 0$. 
Take any $p > 2 \cdot N_{f,g}$. Then $p \nmid w= \gcd (a_0, b_0)$, since $a_0, b_0 \le N_{f,g}$.
Calculate $e_p(x)=\gcd\big(f_p (x),g_p (x)\big)$ in $\Z_p[x]$ by Euclidean algorithm. Then choose $t$ so that (\ref{fold}) holds. 
Construct $k(x)$ applying  Algorithm~\ref{Alg reconstruct} to $t \cdot e_p(x)$. If the primitive part $d(x)= \pp{k(x)}$ divides both $f(x)$ and $g(x)$, then the $\gcd$ for these primitive polynomials if found: $d(x)= \gcd\big(f(x), g(x)\big)$. That follows from consideration about divisor degrees above: if $f(x), g(x)$ had a common divisor $h(x)$ of degree greater than $\deg d(x)$, then, since the degree of $h(x)$ is not altered by $\varphi_p$, we would get $\deg h_p(x)  > \deg d_p(x) = \deg d(x)= \deg e_p(x)$, which contradicts the maximality of $\deg d_p(x)$ by Lemma~\ref{highest degree is gcd for fields}. 

This means that if for $p > 2 \cdot N_{f,g}$ we get $d(x) \nmid f(x)$ or $d(x) \nmid g(x)$,  we have the case when $p$ divides the resultant $R$. Then we just ignore the calculated polynomial, choose another $p > 2 \cdot N_{f,g}$ and redo the steps for it. Repeating these steps for finitely many times, we will eventually arrive to the correct $d(x)$ for the primitive polynomials $f(x), g(x)$. 

The case of arbitrary non-zero polynomials can easily be reduced to this. By arguments mentioned earlier: we should calculate $d(x)$ for primitive polynomials  $\pp{f(x)}$ and $\pp{g(x)}$, and then output the final answer as $r \cdot d(x)$, where $r$ is defined by (\ref{r}). The process we described is the traditional form of the Big prime modular $\gcd$ algorithm.

\begin{Remark}
\label{my remark}
Since our approach in Section~\ref{four problems} evolved the maximality of degrees of the common divisors, we can shorten some of the steps of our algorithm. 
Let us store in a variable, say, $D$ the minimal value for which we already know it is not the $\deg \gcd\big(f(x), g(x)\big)$. As an initial $D$ we may take, say, $D = \min\{\deg f(x), \deg g(x)\}+1$.
Each time we calculate $e_p(x) = \gcd\big(f_p (x),g_p (x)\big)$, check if $\deg e_p(x)$ is equal to or larger than the current $D$.  If yes, we already know that we have an ``inappropriate'' $p$. Then we no longer need use Algorithm~\ref{Alg reconstruct} to reconstruct $k(x)$ and to get $d(x)= \pp{k(x)}$. We just skip these steps and proceed to the next $p$.
Reconstruct $d(x)$ and check if $d(x) | f(x)$ and $d(x) | g(x)$ {\it only} when $\deg e_p(x) < D$. 
Then, if $d(x)$ does not divide $f(x)$ or $g(x)$, we have discovered a new bound $D$ for $\deg \big(\gcd(f(x), g(x)\big)$. So set $D = \deg e_p(x)$ and proceed to the next $p$. If in next step we get $\deg{e_p(x)}\ge D$, we will again be aware that the steps of reconstruction of $d(x)$ need be skipped.
\end{Remark}

We constructed the following algorithm:

\begin{Algorithm}[Big prime modular $\gcd$ algorithm]
\label{BPN}
{\it $\phantom{.}$
\\
Input: non-zero polynomials $f(x), g(x) \in \Z[x]$.\\
Calculate their greatest common divisor $\gcd\big(f(x), g(x)\big)\in \Z[x]$.
}

{\parskip=1mm

\noindent
{\parskip=0mm
01. Calculate $\cont{f(x)}$, $\cont{g(x)}$ in Euclidean domain $\Z$, choose their signs 

\hskip3mm
so that $\sgn{\cont{f(x)}} = \sgn{\lc{f(x)}}$
and $\sgn{\cont{g(x)}} = \sgn{\lc{g(x)}}$.
}

\noindent
02. Set $f(x)= \pp{f(x)}$ and $g(x)= \pp{g(x)}$.

\noindent
03. Calculate $r$ in Euclidean domain $\Z$ by (\ref{r}).

\noindent
{\parskip=0mm
04. Set $a_0= \lc{f(x)}$ and $b_0= \lc{g(x)}$ (they are positive by our selection of signs 

\hskip3mm
for $\cont{f(x)}$ and $\cont{g(x)}$).
}

\noindent
05. Calculate the positive $w=\gcd(a_0,b_0)$ in Euclidean domain $\Z$.

\noindent
06. Set $D = \min\{\deg f(x), \deg g(x)\}+1$.

\noindent
07. Compute the Landau-Mignotte bound $N_{f,g}$ by (\ref{N_fg}).

\noindent
08. Choose a new prime number $p > 2 \cdot N_{f,g}$.

\noindent
09. Apply the reduction $\varphi_p$ to calculate the modular images $f_p (x), g_p (x) \in \Z_p [x]$.

\noindent
10. Calculate $e_p(x) = \gcd\big(f_p (x), g_p (x)\big)$ in the Euclidean domain $\Z_p[x]$.

\noindent
11. If $D \le \deg e_p(x)$

\noindent
12. \hskip5mm 
go to step 08;

\noindent
13. else

\noindent
14.  \hskip5mm
choose a $t$ such that the $\lc{t \cdot e_p(x)} = w$;

\noindent
15.  \hskip5mm
call Algorithm~\ref{Alg reconstruct} to calculate the preimage $k(x)$ of $t \cdot e_p(x)$;

\noindent
16.  \hskip5mm
calculate $\cont {k(x)}$ in Euclidean domain $\Z$;

\noindent
17.  \hskip5mm
set $d(x)= \pp{k(x)} = k(x)/ \cont {k(x)}$;

\noindent
18. \hskip13mm 
if $d(x) | f(x)$ and $d(x) | g(x)$

\noindent
19. \hskip22mm 
go to step 23;

\noindent
20. \hskip13mm 
else

\noindent
21. \hskip22mm 
set $D = \deg e_p(x)$;

\noindent
22. \hskip22mm 
go to step 08.

\noindent
23. Output the result: $\gcd\big(f(x), g(x)\big) = r \cdot d(x)$.
}
\end{Algorithm}

\vskip1mm

Turning back to Remark~\ref{my remark}, notice that  for some prime numbers $p$ we skip the steps 14 -- 18 of Algorithm~\ref{BPN}, and directly jump to the step 08. In fact, Remark~\ref{my remark} has mainly theoretical purpose to display how usage of UFD properties and comparison of divisor degrees may reduce some of the steps of the Big prime modular $\gcd$ algorithm. In practical examples the set of primes we use contains few primes dividing $R = \res{f(x)/d(x), g(x)/d(x)}$, so we may not frequently get examples where the steps 14 -- 18 are skipped. 

\begin{Example}
\label{apply BPA on Knuth}
Let us apply Algorithm~\ref{BPN} to polynomials (\ref{Knuth}) mentioned in Knuth's example above. Since $\lVert f(x) \rVert = \sqrt{113}$ and $\lVert g(x) \rVert = \sqrt{570}$,
$$
N_{f,g}=
 2^{\min\{8,6\}} \cdot \gcd(1,3) \cdot
\min
\left\{
{{\sqrt{113}}\over{1}},
{{\sqrt{570}}\over{3}}
\right\}<512. 
$$
And we can take the prime $p = 1031 > 2 \cdot N_{f,g}$. It is not hard to compute that $\gcd\big(f_{1031} (x), g_{1031} (x)\big) \approx 1$. So $f(x)$ and $g(x)$ are coprime. It is worth to compare $p = 1031$ with much smaller values $p=67$ and $p=37$ obtained below for the same polynomials (\ref{Knuth}) in Example~\ref{apply modified on Knuth} using the modified Algorithm~\ref{BPN-e}. 
\end{Example}

In~\cite{ImGirk} we also apply Algorithm~\ref{BPN} to other polynomials with cases when the polynomials are not coprime.

\section{Estimating the prime divisors of the resultant}
\label{estimating the resultant}

Although at the start of the Big prime modular $\gcd$ algorithm
we cannot compute the resultant $R=\res{f(x)/d(x),g(x)/d(x)}$ 
for the given $f(x),g(x) \in \Z[x]$
(we do not know $d(x)$), we can nevertheless estimate the value of $R$ and the number of its prime divisors. Denote:
\begin{equation}
\label{estimate resultant bound}
\begin{array}{ll} 
 \,\,\,\,\,\,\, 
A_{f,g}   & \!\!\! =
\sqrt{
(n+1)^m (m+1)^n  
}
\cdot N_f^m N_g^n \\ 
  & 
\!\!\! = 
2^{2nm - n - m}\,
\vphantom{ b^{b^{b^b}} }
\sqrt{
(n+1)^m (m+1)^n  
}
\cdot 
{\lVert f(x) \rVert}^m 
{\lVert g(x) \rVert}^n . 
\end{array}
\end{equation}

\begin{Lemma}
\label{estimate resultant}
For any polynomials $f(x), g(x) \in \Z[x]$ and for any of their common divisors $d(x)$ the following holds:
$$ 
|\,\res{f(x)/d(x),g(x)/d(x)}| = |S_{f/d,\, g/d}|  \le A_{f,g}.
$$
\end{Lemma}

\begin{proof}
By Corollary~\ref{N_f corollary} the coefficients of fractions $f(x)/d(x)$ and $g(x)/d(x)$ are bounded, respectively, by $N_f=2^{n-1} \lVert f(x) \rVert$ and $N_g=2^{m-1} \lVert g(x) \rVert$,
where $n = \deg{f(x)}$, $m = \deg{g(x)}$. 
Since the numbers of summands in these fractions are at most $n+1$ and $m+1$, respectively, we get:
$$
\text{
$
\lVert f(x)/d(x) \rVert \le 
\sqrt{(n+1) \, N_f^2}
$
\,\, and \,\,
$
\lVert g(x)/d(x) \rVert \le 
\sqrt{(m+1) \, N_g^2}.
$
}
$$
Applying the Hadamard's maximal determinant bound~\cite{MCA} to the Sylvester matrix $S_{f/d,\, g/d}$, we get that 
$$
|R|= |S_{f/d,\, g/d}| \le 
\left(\! \sqrt{(n+1) }\, N_f \right)^{\! m}
\!\!\! \cdot
\left(\! \sqrt{(m+1) }\, N_g \right)^{\! n}  \!\!\!.
$$
\end{proof}

The bound of (\ref{estimate resultant bound}) is very rough. To see this apply it to the polynomials (\ref{Knuth}) of Knuth's example:
\begin{Example}
\label{apply bound to Knuth}
For polynomials (\ref{Knuth}) we have 
$\lVert f(x \rVert = \sqrt{113}$ and 
$\lVert g(x \rVert = \sqrt{570}$.
So we can estimate $N_{f}< 1408$, $N_{g}< 768$ and $N_{f,g}< 512$. Thus:
$$
|R| \le \sqrt{(8+1)^6 (6+1)^8} \cdot 1408^6 \cdot  570^8 = \omega =
1.6505374299582118582810249858265e+48,
$$
which is a too large number to comfortably operate with.
\end{Example}

\begin{Remark}
\label{why not to chose}
If in  Algorithm~\ref{BPN} we use a prime 
\begin{equation}
\label{too big bound}
p > 2 \cdot A_{f,g},
\end{equation}
then we will get that $p \nmid R = \res{f(x)/d(x),g(x)/d(x)}$ whatever the greatest common divisor $d(x)$ be. 
And, clearly,  $p \nmid w$  holds for $w = \gcd(a_0,b_0)$.
So in this case Algorithm~\ref{BPN} will output the correct $\pp{k(x)}$ using just one $p$, and we will not have to take another $p \nmid w$ after step 18. However, Example~\ref{apply bound to Knuth} shows why it is {\it not} reasonable to chose $p$ by the rule (\ref{too big bound}) to have in  Algorithm~\ref{BPN} one cycle only: it is easier to go via a few cycles for smaller $p$'s rather than to operate with a huge $p$, which is two times larger than the bound $\omega$ obtained in Example~\ref{apply bound to Knuth}.
\end{Remark}

Nevertheless, the bound $A_{f,g}$ may be useful if we remember that the process in Algorithm~\ref{BPN} concerned not the value of $\res{f(x)/d(x),g(x)/d(x)}$ but the {\it number of its distinct prime divisors}.
Let us denote by $p_k\#$ the product of the first $k$ primes: $p_k\# = p_1 \cdot p_2 \cdots p_k$ (where $p_1=2$, $p_2=3$, etc.). They sometimes call $p_k\#$ the {\it ``$k$'th primorial''}. 
The following is essential:

\begin{Lemma}
\label{primorial is largest}
The number of pairwise distinct prime divisors of a positive integer $n$ is less or equal to $\max \left\{k \, | \,  p_k\# \le n \right\}$.
\end{Lemma}

From Lemma~\ref{estimate resultant} and Lemma~\ref{primorial is largest} easily follows:

\begin{Corollary}
\label{primes dividing resultant}
For any polynomials $f(x), g(x) \in \Z[x]$ and for any of their common divisors $d(x)$ the number of pairwise distinct prime divisors of $\res{f(x)/d(x),g(x)/d(x)}$ is at most $k$, where $k$ is the largest number for which $p_k\# \le A_{f,g}$.
\end{Corollary}

Primorial (as a function on $k$) grows very rapidly. Say, for $k=10$ it is more than six billions: $p_{10}\# = 6,469,693,230$. This observation allows to use the bound $A_{f,g}$ in the following way: 
although the value of $A_{f,g}$ as a function on $n=\deg{f(x)}$,  $m=\deg{g(x)}$ and on the coefficients of $f(x)$ and $g(x)$ grows rapidly, the number of its distinct prime divisors, may not be ``very large'' thanks to the fact that $p_k\#$ also grows rapidly. Consider this on polynomials and values from Example~\ref{apply bound to Knuth}:

\begin{Example}
\label{enough is 31}
It is easy to compute that:
$$
p_{30}\#=3.1610054640417607788145206291544e+46 < \omega
$$
and
$$
p_{31}\#=4.014476939333036189094441199026e+48 > \omega,
$$
where $\omega$ is the large number from Example~\ref{apply bound to Knuth}. This means that the number of prime divisors of $R = \res{f(x)/d(x),g(x)/d(x)}$, whatever the divisor $d(x)$ be, is not greater than $30$. And whichever $30+1 = 31$ distinct primes we take, at least one of them will {\it not} be a divisor of $R$. That is, Algorithm~\ref{BPN} for the polynomials of Knuth's example will output the correct answer in not more than $31$ cycles. We {\it cannot} find $31$ primes $p \nmid w$ so that Algorithm~\ref{BPN} arrives to a wrong $d(x)= \pp{k(x)}$ on step 18 for all of them.
\end{Example}

\begin{Remark}
\label{Let us stress}
Let us stress that estimates on the number of prime divisors of the resultant and the analog of Algorithm~\ref{coprime algoritm} can be found elsewhere, for example, in~\cite{MCA}. So the only news we have is that here we use a slightly better value for $N_f$ and $N_g$ to get $2^{n+m}$ times smaller bound for $A_{f,g}$. Namely, in Corollary~\ref{N_f corollary} we estimate $| c_i |$ not by $2^{n} \lVert f(x) \rVert$ but by $2^{n-1} \lVert f(x) \rVert$ (see (\ref{N_f}) and Remark~\ref{smaller degree}). 
This makes the bound $A_{f,g}$ in formula (\ref{estimate resultant bound}) $2^{n+m}$ times lower, since $N_f$ and $N_g$ appear $m$ and $n$ times respectively.
\end{Remark}

\section{An algorithm to check  coprime polynomials}
\label{coprime polynomials}

The first application of the bounds found in previous section is an algorithm checking if the given polynomials $f(x), g(x) \in \Z[x]$ are coprime. Present the polynomials as $f(x)= \cont{f(x)} \cdot \pp{f(x)}$ and $g(x)= \cont{g(x)} \cdot \pp{g(x)}$. If $r = \gcd \big(\cont{f(x)},\cont{g(x)}\big) \not\approx 1$, then $f(x), g(x)$ are not coprime, and we do not have to check the primitive parts, at all.

If $r \approx 1$, then switch to the polynomials $f(x)= \pp{f(x)}$ and $g(x)= \pp{g(x)}$. 
By Corollary~\ref{primes dividing resultant} the number of distinct prime divisors of $\res{f(x)/d(x),g(x)/d(x)}$ is less or equal to $k$, where $k$ is the largest number for which $p_k\# \le A_{f,g}$.

Consider any $k+1$ primes $p_1, \ldots, p_{k+1}$, each not dividing $w = \gcd(a_0, b_0)$, where $a_0 = \lc{f(x)}$ and $b_0 = \lc{g(x)}$. 
If $\gcd \big(f_{p_i}(x), g_{p_i}(x)\big)=1$ for at least one $p_i$, then $f(x)$ and $g(x)$ are coprime because $0 = \deg \gcd \big(f_{p_i}(x), g_{p_i}(x)\big)\ge \deg \gcd \big(f(x), g(x)\big)$ and $r \approx 1$.

And if $\gcd \big(f_{p_i}(x), g_{p_i}(x)\big) \not= 1$ for all $i = 1, \ldots, k+1$, then $f_{p_i}(x)$ and $g_{p_i}(x)$ are not coprime for at least one $p_i$, which is not dividing $\res{f(x)/d(x),g(x)/d(x)}$. This means that $f(x)$ and $g(x)$ are not coprime. We got the following algorithm:

\begin{Algorithm}[Coprime polynomials detection modular algorithm]
\label{coprime algoritm} 
{\it $\phantom{.}$
\\
Input: non-zero polynomials $f(x), g(x) \in \Z[x]$.\\
Detect if $f(x)$ and $g(x)$ are coprime.
}

{\parskip=1mm

\noindent
01. Calculate $\cont{f(x)}$, $\cont{g(x)}$ in Euclidean domain $\Z$.

\noindent
02. Calculate $r$ in Euclidean domain $\Z$ by (\ref{r}).

\noindent
03. If $r \not\approx 1$

\noindent
04. \hskip5mm 
output the result: $f(x)$ and $g(x)$ are not coprime and stop.

\noindent
05. Set $a_0= \lc{f(x)}$ and $b_0= \lc{g(x)}$.

\noindent
06. Calculate $w=\gcd(a_0,b_0)$ in Euclidean domain $\Z$.

\noindent
07. Set $f(x)= \pp{f(x)}$ and $g(x)= \pp{g(x)}$.

\noindent
08. Compute the bound $A_{f,g}$ for polynomials $f(x), g(x)$ by (\ref{estimate resultant bound}).

\noindent
09. Find the maximal $k$ for which $p_k\# \le A_{f,g}$.

\noindent
10. Set $i=1$.

\noindent
11. While $i \not= k+1$ 

\noindent
12. \hskip5mm 
choose a new prime $p \nmid w$;

\noindent
13. \hskip5mm 
apply the reduction $\varphi_{p}$ to calculate the modular images $f_{p} (x), g_{p} (x) \in \Z_{p} [x]$;

\noindent
14. \hskip5mm 
calculate $e_{p} = \gcd\big(f_{p} (x), g_{p} (x)\big)$ in Euclidean domain $\Z_p[x]$;

\noindent
15. \hskip13mm 
if $\deg e_{q_i} = 0$

\noindent
16. \hskip22mm 
output the result: $f(x)$ and $g(x)$ are coprime and stop.

\noindent
17. \hskip5mm
set $i=i+1$.

\noindent
18. If $i < k+1$ 

\noindent
19. \hskip5mm 
go to step 12.

\noindent
20. else

\noindent
21. \hskip5mm 
output the result: $f(x)$ and $g(x)$ are not coprime.
}
\end{Algorithm}

\vskip3mm
Two important advantages of this algorithm are that here we use much smaller primes $p$ (we just require $p \nmid w$, not $p > 2 \cdot N_{f,g}$), and in Algorithm~\ref{coprime algoritm}, unlike in Algorithm~\ref{BPN}, we never need to find $t$, to compute the preimage $k(x)$ of $t\cdot \gcd\big(f_p (x), g_p (x)\big)$ and the primitive part $\pp{k(x)}$.

\begin{Remark}
\label{Knuth sais}
As it is mentioned by Knuth in~\cite{Knuth 2}, in a probabilistic sense the polynomials are much more likely to be coprime than the integer numbers. So it is reasonable to first test by Algorithm~\ref{coprime algoritm} if the given polynomials $f(x), g(x)$ are coprime, and only after that apply Algorithm~\ref{BPN} to find their $\gcd$ in case if they are not coprime. See also Algorithm~\ref{BPN+}, where we combine both these approaches with a better bound for prime $p$.
\end{Remark}

\begin{Example}
\label{Knuth for p=2}
Apply Algorithm~\ref{coprime algoritm} to polynomials (\ref{Knuth}) from Knuth's example. As we saw in Example~\ref{enough is 31}, $k=30$.
For $p=2$ we get $f_2(x)=x^8+x^6+x^4+x^3+1$, $g_2(x)=x^6+x^4+x+1$, which are not coprime, since $\gcd\big(f_2 (x), g_2 (x)\big)=x^2+x+1\not= 1$
And for $p=3$ we get $f_3(x)=x^8+x^6+2x^2+2x+1$, $g_3(x)=2x^4+2x^2$, which are coprime. So $\gcd\big(f(x), g(x)\big)=1$
\end{Example}

\begin{Example}
If $f(x)=x^2+2x+1$ and $g(x)=x+1$. Then $N_f=2\sqrt{6}$, $N_g=\sqrt{2}$ and $A_{f, g}<39$. Since $2 \cdot 3 \cdot 5 \cdot 7 = 210 > 39$, we get that $k=3$, and $\gcd\big(f(x), g(x)\big)\not= 1$  if $\gcd\big(f_p (x), g_p (x)\big)\not= 1$ for any {\it four} primes (not dividing $w$). It is easy to check that $\gcd\big(f_p (x), g_p (x)\big)\not=1$ for $p=2, 3, 5, 7$.
\end{Example}

\section{Other modifications of algorithms}
\label{other algorithms}

The bounds mentioned in Section~\ref{estimating the resultant} can be applied to obtain modifications of Algorithm~\ref{BPN}. Let us outline four ideas, of which only the last two will be written down as algorithms.

For the non-zero polynomials $f(x),g(x) \in \Z[x]$ let us again start by computing  $r = \gcd\big(\cont{f(x)}, \cont{g(x)}\! \big)$ and switching to the primitive parts $f(x)= \pp{f(x)}$ and $g(x)= \pp{g(x)}$, assuming that their leading coefficients $a_0$ and $b_0$ are positive.
Calculate $N_f$, $N_g$ by Corollary~\ref{N_f corollary}, $N_{f,g}$ by Corollary~\ref{N_fg corollary} and $A_{f,g}$ by (\ref{estimate resultant bound}). Find the maximal $k$ for which $p_{k}\# \le A_{f,g}$. Then take any $k+1$ primes $p_1,\ldots,p_{k+1}$ each greater than $2 \cdot N_{f,g}$.
We do not know $d(x)$, but we are aware that the number of prime divisors of $R=\res{f(x)/d(x),g(x)/d(x)}$ is less than equal to $k$. So  at least one of the primes  $p_1,\ldots,p_{k+1}$ is not dividing $R$. To find it compute the degrees of $e_{p_i}(x)$ for all $i=1,\ldots, k+1$. Take any $p_i$, for which $\deg{e_{p_i}(x)}$ is the minimal (in case there are more than one $p_i$'s with this property, take one of them, preferably, the smallest of all).

By our construction, $\deg{e_{p_i}(x)} = \deg \gcd\big(f(x), g(x)\big)$ holds. So we can proceed to the next steps: choose a $t$, such that $\lc{t \cdot e_{p_i}(x)} = w=\gcd(a_0, b_0)$; then find by Algorithm~\ref{Alg reconstruct} the pre-image $k(x)$ of $t \cdot e_{p_i}(x)$; then proceed to its primitive part $d(x) = \pp{k(x)}$; and then output the final answer as $r \cdot d(x)$.

The advantage of this approach is that we do not have to go via the steps 14--18 of Algorithm~\ref{BPN} for more than one prime $p$. Also, we do not have to take care of the variable $D$. But the disadvantage is that we have to compute $e_{p_i}(x)$ for large primes for $k+1$ times (whereas in Algorithm~\ref{BPN} the correct answer could be discovered after consideration of fewer primes).
Clearly, the disadvantage is a serious obstacle, since repetitions for $k+1$ large primes consumes more labour than the steps 14--18 of Algorithm~\ref{BPN}. So this is just a theoretical idea, not an approach for an effective algorithm.

\vskip5mm
The disadvantage can be reduced in the following way: in previous arguments, after we find $p_{k}\#$ and $k$, select the prime numbers $p_1,\ldots,p_{k+1}$ each satisfying the condition $p_{i} \nmid w$. This is a much weaker condition than the condition $p_{i} > 2 \cdot N_{f,g}$ used above, so we will surely get smaller primes. Take $M$  to be the minimum of all degrees $\deg{e_{p_i}(x)}$ for all $i=1,\ldots, k+1$. 
Since none of the primes $p_1,\ldots,p_{k+1}$ divides $w$, for any $i=1,\ldots,k+1$ we have $\deg{e_{p_i}(x)} \ge \deg{\gcd\big(f(x), g(x)\big)}$.
On the other hand, since at least one of the primes $p_1,\ldots,p_{k+1}$ does not divide $R$, we know that for that $p_i$ the degree of $e_{p_i}(x)$ is equal to $\deg{\gcd\big(f(x), g(x)\big)}$. Combining these we get that  $\deg{\gcd\big(f(x), g(x)\big)} = M$. 
Since we know $M$, we can take a prime $p  > 2 \cdot N_{f,g}$, compute the $e_{p}(x)$ and check its degree: if $\deg{e_{p}(x)}\not= M$, then we have a wrong $p$ (we no longer need go the steps 14--18 of Algorithm~\ref{BPN} to discover that). Then choose a new value for $p$ and repeat the step. And if  $\deg{e_{p}(x)}= M$, then we have the right $p$. We calculate $t$, the pre-image $k(x)$ of $t \cdot e_{p_i}(x)$, then $d(x) = \pp{k(x)}$, and output the answer $r \cdot d(x)$ (see Algorithm~\ref{BPN+} for a better version of this idea).

\vskip5mm
The third modification, not depending on $A_{f,g}$ can be constructed by estimating $\deg \gcd\big(f(x), g(x)\big) = \deg d(x)$ by means of an auxiliary prime number $q$. By Landau-Mignotte Theorem~\ref{Landau Mignotte}, if $h(x)=c_0 x^k+\cdots +c_k$ is any divisor of the polynomials $f(x)=a_0 x^n+\cdots +a_n$ and $g(x)=b_0 x^m+ \cdots +b_m$
then 
$
| c_i | \le 2^{k} |c_0 / a_0| \lVert f(x) \rVert
$
and
$
| c_i | \le 2^{k} |c_0 / b_0| \lVert g(x) \rVert
$.
Since $|c_0 / a_0|$ is bounded by $|\gcd(a_0,b_0)/ a_0|$ and $|c_0 / b_0|$ is bounded by $|\gcd(a_0,b_0)/ b_0|$, we get the following analog of (\ref{N_fg}):
\begin{equation}
\label{common bound for k}
| c_i | \le 
 2^{k} \cdot \gcd(a_0, b_0) \cdot
\min
\left\{
{{\lVert f(x) \rVert}\over{|a_0|}},
{{\lVert g(x) \rVert}\over{|b_0|}}
\right\}.
\end{equation}

Now assume $q$ is a prime not dividing $w$, and denote $s(q,f,g) = \deg \gcd\big(f_q (x),g_q (x)\big)$. By Corollary~\ref{resultant fold p},
$\deg d(x) \le s(q,f,g)$. We get for the coefficients of $d(x)$ the following bound : $| c_i | \le M_{q,f,g}$, where 
\begin{equation}
\label{Mqfg}
M_{q,f,g}=
 2^{s(q,f,g)} \cdot \gcd(a_0, b_0) \cdot
\min
\left\{
{{\lVert f(x) \rVert}\over{|a_0|}},
{{\lVert g(x) \rVert}\over{|b_0|}}
\right\}.
\end{equation}
$M_{q,f,g}$ is a better bound for the coefficients of $d(x)$ because $2^{s(q,f,g)}$ may be considerably less than $2^{\min\{n,m \}}$\!.

We can improve Algorithm~\ref{BPN}, if we preliminarily find $s(q,f,g)$ by calculating the $\gcd\big(f_q (x),g_q (x)\big)$ for an ``auxiliary'' prime $q \nmid w$, and then chose the ``main'' prime $p$ by the rule $p > 2 \cdot M_{q,f,g}$ (instead of $p > 2 \cdot N_{f,g}$). Observe that if $p > 2 \cdot M_{q,f,g}$ then also $q \nmid w = \gcd(a_0, b_0)$ because $\min
\left\{
{{\lVert f(x) \rVert}\over{|a_0|}},
{{\lVert g(x) \rVert}\over{|b_0|}}
\right\} \ge 1$.
Additionally, we can introduce the variable $D$ to store those values $\deg \gcd\big(f_p (x),g_p (x)\big)$ that we know are greater than $\deg d(x)$. We get the following algorithm:

\begin{Algorithm}[Big prime modular $\gcd$ algorithm with a preliminary estimate on divisor degree]
\label{BPN-e}
{\it $\phantom{.}$
\\
Input: non-zero polynomials $f(x), g(x) \in \Z[x]$.\\
Calculate their greatest common divisor $\gcd\big(f(x), g(x)\big)\in \Z[x]$.
}

{\parskip=1mm

\noindent
{\parskip=0mm
01. Calculate $\cont{f(x)}$, $\cont{g(x)}$ in Euclidean domain $\Z$, choose their signs 

\hskip3mm
so that $\sgn{\cont{f(x)}} = \sgn{\lc{f(x)}}$
and $\sgn{\cont{g(x)}} = \sgn{\lc{g(x)}}$.
}

\noindent
02. Set $f(x)= \pp{f(x)}$ and $g(x)= \pp{g(x)}$.

\noindent
03. Calculate $r$ in Euclidean domain $\Z$ by (\ref{r}).

\noindent
{\parskip=0mm
04. Set $a_0= \lc{f(x)}$ and $b_0= \lc{g(x)}$ (they are positive by our selection of signs 

\hskip3mm
for $\cont{f(x)}$ and $\cont{g(x)}$).
}

\noindent
05. Calculate the positive $w=\gcd(a_0,b_0)$ in Euclidean domain $\Z$.

\noindent
06. Set $D = \min\{\deg f(x), \deg g(x)\}+1$.

\noindent
07. Choose a prime number $q \nmid w$.

\noindent
08. Apply the reduction $\varphi_q$ to calculate the modular images $f_q (x), g_q (x) \in \Z_q [x]$.

\noindent
09. Calculate $e_q(x) = \gcd\big(f_q (x), g_q (x)\big)$ in the Euclidean domain $\Z_q[x]$.

\noindent
10. Set $s(q,f,g) = \deg{e_q(x)}$.

\noindent
11. Calculate $M_{q,f,g}$ by  (\ref{Mqfg}) using the value of $s(q,f,g)$.

\noindent
12. Choose a new prime number $p > 2 \cdot M_{q,f,g}$.

\noindent
13. Apply the reduction $\varphi_p$ to calculate the modular images $f_p (x), g_p (x) \in \Z_p [x]$.

\noindent
14. Calculate $e_p(x) = \gcd\big(f_p (x), g_p (x)\big)$ in Euclidean domain $\Z_p[x]$.

\noindent
15. If $D \le \deg e_p(x)$

\noindent
16. \hskip5mm
go to step 12. 

\noindent
17. else

\noindent
18.  \hskip5mm
choose a $t$ such that the $\lc{t \cdot e_p(x)} = w$;

\noindent
19.  \hskip5mm
call Algorithm~\ref{Alg reconstruct} to calculate the preimage $k(x)$ of $t \cdot e_p(x)$;

\noindent
20.  \hskip5mm
calculate $\cont {k(x)}$ in Euclidean domain $\Z$;

\noindent
21.  \hskip5mm
set $d(x)= \pp{k(x)} = k(x)/ \cont {k(x)}$;

\noindent
22. \hskip13mm 
if $d(x) | f(x)$ and $d(x) | g(x)$

\noindent
23. \hskip22mm 
go to step 27;

\noindent
24. \hskip13mm 
else

\noindent
25. \hskip22mm 
set $D = \deg e_p(x)$;

\noindent
26. \hskip22mm 
go to step 12.

\noindent
27. Output the result: $\gcd\big(f(x), g(x)\big) = r \cdot d(x)$.

}
\end{Algorithm}

\begin{Example}
\label{apply modified on Knuth}
Let us apply Algorithm~\ref{BPN-e} to polynomials (\ref{Knuth}) from Knuth's example. Since $w = 1$, take $q=2$. We have already computed in Example~\ref{Knuth for p=2} that $e_2(x) = \gcd\big(f_2 (x), g_2 (x)\big)=x^2+x+1$. Then $s(2,f,g) = \deg e_2(x)=2$ and 
$$
M_{2,f,g}=
 2^2 \cdot 1 \cdot
\min
\left\{
{{\sqrt{113}}\over{1}},
{{\sqrt{570}}\over{3}}
\right\}<31.84. 
$$
Take $p=67 > 2 \cdot M_{2,f,g}$.
It is easy to calculate that $\gcd\big(f_{67} (x), g_{67} (x)\big) \approx 1$. Compare this with Example~\ref{apply BPA on Knuth}, where we had to use much larger prime $p=1031$.
Moreover, if we take as an auxiliary $q$, say, $q=3$, then $s(3,f,g) = \deg e_3(x)=0$ and 
$
M_{3,f,g}\le 15.92. 
$
So we can take an even smaller prime $p=37 > 2 \cdot M_{3,f,g}$.
\end{Example}

\vskip4mm
The ideas of  Algorithm~\ref{coprime algoritm} and of Algorithm~\ref{BPN-e} can be combined to work with more than one auxiliary prime $q$. 
Like we mentioned in Remark~\ref{Knuth sais}, Knuth in~\cite{Knuth 2} recommends to first check if the polynomials  $f(x), g(x) \in \Z[x]$ are coprime, and to proceed to their $\gcd$ calculation only after we get that they are not coprime (this is motivated by probabilistic arguments).
Compute $A_{f,g}$ by formula (\ref{estimate resultant bound}) and find a $k$ like we did it in step 09 of  Algorithm~\ref{coprime algoritm}: $k$ is the maximal number for which  $p_k\# \le A_{f,g}$. Then choose any $k+1$  primes $q_1,\ldots , q_{k+1}$ not dividing $w$, and start computing the modular $\gcd$'s $e_{q_1}(x), e_{q_2}(x), \ldots$ ($k+1$ times). If at some step we find $\deg{e_{q_i}(x)} = 0$, then we are done:
the polynomials $f(x), g(x)$ are coprime if $r \approx 1$,
or their $\gcd$ is the non-trivial scalar $r \not\approx 1$.
And if $\deg{e_{q_i}(x)}> 0$ for all $q_i$, then we know that:
\begin{enumerate}
  \item these polynomials are not coprime, and
  \item the positive degree of $\gcd \big( f(x), g(x) \big)$ is the minimum
  \begin{equation}
\label{minimum of q}
s(f,g) = \min
\left\{
\deg{e_{q_1}(x)},  \ldots , \deg{e_{q_{k+1}}(x)}
\right\}>0.
\end{equation}
\end{enumerate}
This exact value of $s(f,g)  = \deg \gcd \big( f(x), g(x) \big)$ is a better result than the estimate $s(q,f,g)$ obtained earlier by just one $q$.

Like above, we can assume $f(x), g(x)$ to be primitive (if not, we can again denote $r=\gcd\big( \cont{f(x)}, \cont{g(x)} \!\big)$ and switch to the primitive parts $f(x) = \pp{f(x)}$ and $g(x) = \pp{g(x)}$).
Applying the Landau-Mignotte Theorem~\ref{Landau Mignotte} for the coefficients $c_i$ of $\gcd \big( f(x), g(x) \big)$, we get that $|c_i| \le M_{f,g}$, where
\begin{equation}
\label{Mfq}
M_{f,g}=
 2^{s(f,g)} \cdot \gcd(a_0, b_0) \cdot
\min
\left\{
{{\lVert f(x) \rVert}\over{|a_0|}},
{{\lVert g(x) \rVert}\over{|b_0|}}
\right\}.
\end{equation}

Now we can take a $p > 2 \cdot M_{f,g}$ and by the Euclidean algorithm calculate $e_{p}(x)$ in $\Z_p[x]$. 
If $\deg{e_{p}(x)} > s(f,g)$, 
we drop this $p$ and choose another prime $p > 2 \cdot M_{f,g}$. 
And if  $\deg{e_{p}(x)} = s(f,g)$, then we proceed to the final steps: we choose the $t$,  then get the preimage $k(x)$ of $t \cdot e_{p}(x)$, then go to the primitive part $d(x) = \pp{k(x)}$ and output the final answer as $\gcd \big( f(x), g(x) \big) = r \cdot \pp{k(x)}$.

\begin{Remark} 
This approach has the following advantages:
Firstly, the bound on primes $p$ is better than formula (\ref{Mqfg}) since here we have not $2^{s(q,f,g)}$ but $2^{s(f,g)}$.
Secondly, we no longer need calculate the number $t$, the preimage $k(x)$ and the primitive part $d(x)$ for {\it more than one} prime $p$. Because, if the selected $p > 2 \cdot M_{f,g}$ is not appropriate, we already have an indicator of that: $\deg{e_{p}(x)} > s(f,g)$.
\end{Remark}

We built the following algorithm:

\begin{Algorithm}[Big prime modular $\gcd$ algorithm with preliminary estimates on divisor degrees by multiply primes]
\label{BPN+}
{\it $\phantom{.}$
\\
Input: non-zero polynomials $f(x), g(x) \in \Z[x]$.\\
Calculate their greatest common divisor $\gcd\big(f(x), g(x)\big)\in \Z[x]$.
}

{\parskip=1mm

\noindent

\noindent
{\parskip=0mm
01. Calculate $\cont{f(x)}$, $\cont{g(x)}$ in Euclidean domain $\Z$, choose their signs 

\hskip3mm
so that $\sgn{\cont{f(x)}} = \sgn{\lc{f(x)}}$
and $\sgn{\cont{g(x)}} = \sgn{\lc{g(x)}}$.
}

\noindent
02. Set $f(x)= \pp{f(x)}$ and $g(x)= \pp{g(x)}$.

\noindent
03. Compute the bound $A_{f,g}$ for polynomials $f(x), g(x)$ by  (\ref{estimate resultant bound}).

\noindent
04. Find the maximal $k$ for which $p_k\# \le A_{f,g}$.

\noindent
05. Calculate $r$ in Euclidean domain $\Z$ by (\ref{r}).

\noindent
{\parskip=0mm
06. Set $a_0= \lc{f(x)}$ and $b_0= \lc{g(x)}$ (they are positive by our selection of signs 

\hskip3mm
for $\cont{f(x)}$ and $\cont{g(x)}$).
}

\noindent
07. Calculate the positive $w=\gcd(a_0,b_0)$ in Euclidean domain $\Z$.

\noindent
08. Set $s(f,g) = \min\{\deg{f(x)},\deg{g(x)}\}$.

\noindent
09. Set $i=1$.

\noindent
10. While $i \not= k+1$ 

\noindent
11. \hskip5mm 
choose a new prime $q_i \nmid w$;

\noindent
12. \hskip5mm 
apply the reduction $\varphi_{q_i}$ to calculate the modular images $f_{q_i} (x), g_{q_i} (x) \in \Z_{q_i} [x]$;

\noindent
13. \hskip5mm 
calculate $e_{q_i} = \gcd\big(f_{q_i} (x), g_{q_i} (x)\big)$ in Euclidean domain $\Z_p[x]$;

\noindent
14. \hskip5mm 
if $\deg e_{q_i} \le s(f,g)$

\noindent
15. \hskip13mm 
set $s(f,g) = \deg e_{q_i}$;

\noindent
16. \hskip13mm 
if $\deg e_{q_i} = 0$

\noindent
17. \hskip22mm 
set $d(x)=1$;

\noindent
18. \hskip22mm 
go to step 32;

\noindent
19. \hskip5mm
set $i=i+1$.

\noindent
20. Calculate $M_{f,g}$ by  (\ref{Mfq}) using the value of $s(f,g)$.

\noindent
21. Choose a new prime number $p > 2 \cdot M_{f,g}$.

\noindent
22. Apply the reduction $\varphi_p$ to calculate the modular images $f_p (x), g_p (x) \in \Z_p [x]$.

\noindent
23. Calculate $e_p(x) = \gcd\big(f_p (x), g_p (x)\big)$ in Euclidean domain $\Z_p[x]$.

\noindent
24. If $\deg e_{p} = s(f,g)$

\noindent
25.  \hskip5mm
choose a $t$ such that the $\lc{t \cdot e_p(x)} = w$;

\noindent
26.  \hskip5mm
call Algorithm~\ref{Alg reconstruct} to calculate the preimage $k(x)$ of $t \cdot e_p(x)$;

\noindent
27.  \hskip5mm
calculate $\cont {k(x)}$ in Euclidean domain $\Z$;

\noindent
28.  \hskip5mm
set $d(x)= \pp{k(x)} = k(x)/ \cont {k(x)}$;

\noindent
29. \hskip5mm
go to step 32;

\noindent
30. else

\noindent
31. \hskip5mm
go to step 21.

\noindent
32. Output the result: $\gcd\big(f(x), g(x)\big) = r \cdot d(x)$.

}
\end{Algorithm}

\begin{Example}
\label{apply plus on Knuth}
Let us apply Algorithm~\ref{BPN+} again on polynomials of Knuth's example~(\ref{Knuth}).
As we saw in Example~\ref{enough is 31}, $k=30$. So we may have to consider at most $31$ auxiliary primes $q_i$. But we in fact need just two of them, because $\deg\gcd\big(f_2(x), g_2(x) \big) = \deg(x^2 +x +1) = 2$ and $\deg\gcd\big(f_3(x), g_3(x) \big) = \deg(1) = 0$ (see Example~\ref{Knuth for p=2}). So in Algorithm~\ref{BPN+} we jump from step 16 to step 32 directly.
\end{Example}

{\footnotesize
\vskip2mm \noindent  Vahagn H. Mikaelian:\\ Informatics and Applied Mathematics Department\\ Yerevan State University\\ Yerevan 0025, Armenia.\\ E-mail: v.mikaelian@gmail.com
}


\begin{thebibliography}{9}

\bibitem{Brown}
W. Brown,
{\it  On Euclid's Algorithm and the Computation of Polynomial Greatest Common Divisors}, 
J. ACM, 1971, 18, 478--504.

\bibitem{Basic algebra}
P. M. Cohn,
{\it  Basic algebra. Groups, rings and fields}, 
Springer-Verlag, London, 2003. 

\bibitem{Introduction to ring theory}
P. M., Cohn
{\it  Introduction to ring theory}, 
Springer-Verlag, London, 2000. 


\bibitem{Davenport S T}
J. H., Davenport, Y. Siret, E. Tournier, 
{\it  Computer algebra. Systems and algorithms for algebraic computation}, 
Second edition
Academic Press, London, 1993. 

\bibitem{Garrett}
P. B. Garrett,
{\it  Abstract algebra}, 
Boca Raton: Chapman \& Hall/CRC, 2008.


\bibitem{Dummit}
D. S. Dummit, R. M.Foote,
{\it  Abstract algebra}, 
Third edition
s.l.:John Wiley and Sons, 2004.


\bibitem{Knuth 2}
D., Knuth,
{\it  The art of computer programming. Vol. 2. Seminumerical algorithms.}, 
Second edition. 
Addison-Wesley Series in Computer Science and Information Processing. s.l.: Addison-Wesley, 1969.

\bibitem{Kostrikin}
A. I. Kostrikin,
{\it  An introduction to algebra. Vol. 2. Main structures.}, 
FizMatLit, Moscow, 2004 (Russian).

\bibitem{Mignotte}
M. Mignotte, 
{\it Mathematics for computer algebra. Translated from the French by Catherine Mignotte},
Springer-Verlag, New York, 1992. 

\bibitem{Lang}
S. Lang, 
{\it Algebra},
Revised third edition. Graduate Texts in Mathematics, 211.. New York: Springer-Verlag., 2002.


\bibitem{ImGirk}
V. H. Mikaelian, 
{\it An introduction to computer algebra},
in preparation.

\bibitem{General UFD's}
V. H. Mikaelian, 
{\it The Big prime gcd algorithm on general UFD's},
in preparation.



\bibitem{Metabelien}
V. H. Mikaelian, 
\textit{Metabelian varieties of groups and wreath products of abelian groups}, 
J. Algebra, 2007 (313), 2, 455--485.  



\bibitem{Pankratev}
E. V. Pankratev, 
{\it Elements of computer algebra Study Guide},
BINOM. LZ, Moscow, 2007. 

\bibitem{MCA}
J. von zur Gathen, J. Gerhard 
{\it Modern Computer Algebra},
Third edition, Cambridge University Press, Cambridge, 2013. 


\end{thebibliography}
\end{document}